\documentclass[a4paper,reqno]{amsart}
\usepackage[T2A]{fontenc}
\usepackage[cp1251]{inputenc}
\usepackage{amssymb,amscd,amsthm,amsaddr}
\usepackage[usenames]{color}
\usepackage{colortbl}
\usepackage{hyperref}
\hypersetup{
  colorlinks = true,
  urlcolor = blue,
  linkcolor = blue,
  citecolor = red
}

\theoremstyle{plain}
\newtheorem{theo}{Theorem}

\newcommand{\supp}{\mathop{\rm supp}\nolimits}

\allowdisplaybreaks

\begin{document}
\centerline {\Huge Orthogonal wavelet bases}

\smallskip

\centerline {\Huge  on generalized Vilenkin groups\footnote{The authors were supported by the Russian Science Foundation (grant 23-11-00178)}}

\bigskip

\centerline {\large M. Babushkin and M. Skopina}

\bigskip

\bigskip
{\bf Abstract} 
{\small
  Wavelet systems on the generalized Vilenkin groups are considered. An algorithmic method for the  construction of orthogonal wavelet bases is presented. 
  These bases consist of compactly supported test functions (i.e. functions whose Fourier transform is also compactly supported).
}



\bigskip

\section{Introduction}

The well known classical Vilenkin group $V_p$, $p=2,3, \ldots$,  is associated with the  cyclic group of order $p$ (see, e.g., \cite {AVDR}, \cite{SWS}, \cite{GES}). Similar structures, associated with an arbitrary finite abelian group $G$ instead of a cyclic group, were recently introduced in~\cite{VS} and called generalized  Vilenkin groups.  Such a group is also  associated with a dilation matrix $M$ whose set of digits is  isomorphic  to $G$.  The  characters for  generalized Vilenkin groups is given by formulas that are very similar to those for the classical Vilenkin groups (with  $M$ instead of $p$).
The harmonic analysis on such groups is also very similar to harmonic analysis on the usual Vilenkin groups. 
In particular, there exists a class of compactly supported functions whose Fourier transform is also compactly supported (called test functions), that may be considered as an analogue of the Schwartz class in the classical analysis. 

Wavelet frames on the classical  Vilenkin groups have been actively studied, e.g., in~\cite{27},~\cite{25}, \cite{FLS}, \cite{FS23}, ~\cite{FMS19}, in particular, related to the orthogonal wavelet bases see~\cite{Vil}, \cite{F05}, \cite{FMS19}, \cite{26}.
Wavelet frames on the generalized  Vilenkin groups were investigated in~\cite{BS}.

The current paper is devoted to the construction of orthogonal wavelet bases on generalized Vilenkin groups. 

\section{Notation and preliminaries}

As usual,  ${\mathbb Z}^d$ denotes the integer lattice of the Euclidean space  ${\mathbb R}^d$, and $\mathbf{0}$ is the  zero vector in ${\mathbb R}^d$. The characteristic function of a set $Y\subset {\mathbb R}^d $ is denoted by $\mathbf{1}_Y$.
An  integer  $d\times d$-matrix $M$ is called a dilation matrix, if  all its  eigenvalues are greater than~1 in absolute value. The class of such matrices will be denoted by  $\mathfrak{M}_d$. It is well known (see, e.g., \cite[\S~2.2]{NPS}) that 
for any matrix $M\in\mathfrak{M}_d$, there holds 
\begin{equation*}
  \sum_{n=1}^{\infty}\|M^{-n}\|^{\delta} < \infty \quad \forall \delta>0. 
\end{equation*}
For  $M\in\mathfrak{M}_d$ we set  $m:=|\det M|$, and  the matrix transposed to $M$ will be denoted by $M^*$. 

Vectors $l,k\in\mathbb{Z}^d$ is said to be {\it congruent modulo} $M$   (write $l \equiv k\,({\rm mod}\, M)$) if $l-k = Ms$ for some $s\in\mathbb{Z}^d$.
For $M\in\mathfrak{M}_d$, a set  $D=D(M)= \{s_0, \,s_1, \dots, s_{m-1}\}$, where  $s_0=\mathbf{0}$, $s_j\in \mathbb{Z}^d$ and $s_i\not\equiv s_j \,({\rm mod}\, M)$ whenever $i\neq j$,  is called a  {\it set of digits} for matrix $M$.
Following~\cite{VS}, let us consider an algebraic structure  $V$, consisting of the sequences $x=\{x_j\}_{j\in\mathbb Z}$, where $x_j\in D(M)$ and $x_j$ are non-zero only for a finite number of negative $j$, equipped 
with  the coordinate-wise addition modulo~$M$ (denoted by $\oplus$). It is easy to see that  $V$ is an abelian group. Denote by $\theta$ the neutral group element. The topology on $V$ is introduced via the complete system of the following neighbourhoods of  $\theta$:  
$$
U_l:=\{x\in V\,:\ x_j=0 \   \mbox{for all}\ j\le l\},  \quad l\in {\mathbb Z}, \quad U:=U_0.
$$
Thus, we have a topological group $V$, that is called {\it generalized Vilenkin group} associated with 
a  matrix $M$ and a set of its digits $D(M)$. It is not difficult to check that $V$ is a locally compact group,
which yields that  there exists 
the unique  Haar measure $\mu$ on $V$, normalized by the equality $\mu U=1$. 
It is known that this measure is invariant under the group operation. The Lebesgue integral $\int_S f d\mu$ is defined in the  standard way.

The  generalized Vilenkin group 
associated with the matrix  $M^*$ and  a set of its digits  $D^* = D(M^*)=\{s^*_0, \,s^*_1, \dots, s^*_{m-1}\}$ 
will be denoted by $V^*$. The sets $U_l^*$, $U^*$
are introduced in $V^*$ similarly to the sets $U_l$, $U$  in $V$, and $\mu^*$ denotes the normalized Haar measure on $V^*$. 

Let now $G$ be an arbitrary finite abelian group  of order $m$. Since there exists a dilation matrix $M$ such that $| \det M|=m$ and  the group of its  digits is isomorphic to $G$ 
(see \cite[Theorem 1]{VS}), the following  is well defined: 
the {\it generalized Vilenkin group associated with $G$ }  is 
the  generalized Vilenkin group 
associated with a matrix  $M$ whose set of digits is isomorphic to $G$. 

It is also proved in~\cite{VS}, that any continuous  character $\mathrm{ch}(x)$  of a generalized Vilenkin group $V$ is associated with some $\omega\in V^*$ and  given by 
$$
\mathrm{ch}(x)=\chi(x,\omega): =  \exp\left({2\pi i}\sum\limits_{j\in {\mathbb Z}} \langle M^{-1}x_j,\omega_{1-j}\rangle\right), \quad x\in V.
$$
It follows that the group of characters of $V$  and the group $V^*$ are isomorphic, and not only as the algebraic groups, but also as topological groups (see~\cite{VS}).

According to the general theory of  harmonic analysis on the topological groups,
the Fourier transform on $f\in  L_1(V)$ is defined by the formula
$$
\hat{f}(\omega) =\int\limits_V f(x)\overline{\chi(x,\omega)}\,d\mu(x), \quad  \omega \in V^{*},
$$
and it extends to the space $ L_2(V)$ in the standard way.
Let  $  \mathcal M$ be the mapping  on $V$ defined as follows: if $x=\{x_j\}_{j=-\infty}^{+\infty}\in V$, then ${ \mathcal M} x$ is the element of $V$ such that 
$({ \mathcal M} x)_j=x_{j+1}$. Let 
$$
H  := \{\omega\in  V\,: \ \omega_j=\mathbf{0} \quad\mbox{for all}\quad j>0 \}.
$$  
For every $ k\in\mathbb{Z}_+$, denote by $\gamma_{[k]}$ the element of $H$ such that   $(\gamma_{[k]})_{-j}=s_{k_j}$, $k_j\in\{0,1,\dots, m-1\}$ and  $	k=\sum_{j=0}^{\infty} k_j m^{j}$. Set also 
$$
U_{n,k}:={\mathcal M}^{-n} \gamma_{[k]} + {\mathcal M}^{-n}( U),   \quad  n\in\mathbb{Z},\  k\in\mathbb{Z}_+. 
$$
The mapping ${\mathcal M}^*$ and the sets  $U_{n,k}^*$, $H^*$   are introduced for $V^*$ in the same way.

Let $ k\in \mathbb{Z}_{+}$, then the function $W_k$  defined on $V^*$ by 
$$ W_k:=\chi(\gamma_{[k]},\cdot), $$ 
is called the {\it Walsh function}, and finite linear combination of Walsh functions $W_0,\dots, W_{m^n-1}$ is called the Walsh polynomial of order $n$. It is proved in~\cite{VS}, that the 
Walsh functions $W_k$ are $H^*$-periodic,  constant on each set $U^*_{n,s}$,  $s \in\mathbb{Z}_+$, $0 \leq k \leq m^{n}-1$, and the Walsh system  $\{ W_k \}_{k=0}^{\infty}$ is an orthonormal basis in $L_{2}(U^{*})$.

Note it is explained in~\cite{VS} that the usual properties of the Fourier transform on $L_2(V)$ hold true, in particular, we have the Plancherel theorem, existence of the inverse Fourier transform and the following formulas
\begin{align}
  &\widehat{f(\cdot\oplus \gamma)}=\widehat f \chi(\gamma,\cdot) \quad \forall  \gamma\in H, \label{Fourier_shift}\\
  &\int\limits_V f({\mathcal M}^j\cdot)\ d\mu=\frac1{m^j}\int\limits_V f \ d\mu,\notag\\
  &\widehat{f({\mathcal M}^j\cdot)} =\frac1{m^j}\widehat f({{\mathcal M}^*}^{-j}\cdot).\notag
\end{align}

Let us introduce the following classes of functions $f$ defined on $V$:
$$
\mathcal{S}_n(V) := \{ f\,: \ f(x)=f(x')\quad\mbox{for all}\quad x, x'\in U_{n,k},
\  k\in {\mathbb Z}_+\},
$$
$$
\mathcal{S}_n^{(k)}(V) := \{ f\in \mathcal{S}_n(V)\,: \ \supp f \subset {\mathcal M}^k(U) \}, \quad
\mathcal{S}(V) := \bigcup_{n,k \in \mathbb{Z}} \mathcal{S}_n^{(k)}(V).
$$
The elements of $\mathcal{S}(V)$ are called the {\it test functions}.

The classes $\mathcal{S}_n(V^*)$,  $\mathcal{S}_n^{(k)}(V^*)$ and $\mathcal{S}(V^*)$  are defined similarly.
\begin{theo}\cite[Theorem 7]{VS}
  \label{p8.2}
  For any integers $n$ and $k$ there holds 
  $$
  f \in \mathcal{S}_n^{(k)}(V) \ \Longleftrightarrow \ \widehat{f} \in \mathcal{S}_k^{(n)}(V^*). 
  $$
\end{theo}

\section{Wavelets on $V$}

A function $\varphi\in L_2(V)$ is called {\it refinable} if there exist numbers $\alpha_k$, $k\in \mathbb{Z}_{+}$, such that
\begin{equation}
  \label{refinable}
  \varphi(x)=\sum_{k=0}^\infty \alpha_k \varphi({\mathcal M}x \oplus \gamma_{[k]}) \quad \forall x\in V.
\end{equation}

\begin{theo}\cite{BS}
  If $\varphi\in L_2(V)$ is a compactly supported refinable function, then
  \begin{equation}
	\widehat \varphi= m_0({{\mathcal M}^*}^{-1}\cdot)\widehat \varphi({{\mathcal M}^*}^{-1}\cdot),
    \label{refinement_eq_Fourier}
  \end{equation}
  where $m_0$ (called the mask) is a Walsh polynomial on $V^*$.
\end{theo}

Let $m_0$ be the mask of a refinable function  $\varphi$.
Suppose that there exist Walsh polynomials $m_1,\dots, m_r$, $r\ge m-1$, such that 
\begin{equation}
  \label{wav}
  \sum_{\nu=0}^rm_\nu({\mathcal M^*}^{-1}(\cdot\oplus s^*_k))\overline{m_\nu({\mathcal M^*}^{-1}(\cdot\oplus s^*_l))}=\delta_{kl}.
\end{equation}
The functions $\psi^{(\nu)}$, $\nu=1,\dots, r$, defined by 
$$
\widehat {\psi^{(\nu)}}= m_\nu({\mathcal M^*}^{-1}\cdot)\widehat \varphi({\mathcal M^*}^{-1}\cdot),\quad \nu=1,\dots, r, 
$$
are called {\it wavelet functions}. The corresponding {\it wavelet system}
$$
\psi_{jk}^{(\nu)}(x):=m^{j/2}\psi^{(\nu)}({\mathcal M}^j x\oplus \gamma_{[k]}), 
\quad \nu=1,\dots, r,\ \ j\in \mathbb Z, \ \ k\in \mathbb Z_+,
$$
is said to be {\it generated by $\varphi$}. Obviously, for the existence of wavelet system generated by $\varphi$,
it is necessary that
\begin{equation}
  \sum_{k=0}^{m-1} |m_0({\mathcal M^*}^{-1}(\omega\oplus s^*_k))|^2\le 1\quad \forall \omega\in V^*.
  \label{aaa}
\end{equation}

\begin{theo}\cite{BS}
  \label{thm:parseval-frame}
  Let $\varphi\in L_2(V)$ be a compactly supported  refinable function and $\widehat\varphi(\theta)=1$. Then a wavelet system 
  $\{\psi^{(\nu)}_{jk}\}_{j,k,\nu}$ generated by $\varphi$ is a Parseval frame in $L_2(V)$, i.e., 
  \begin{equation}
    \label{pars}
	\sum_{\nu=1}^r\sum_{j\in \mathbb Z}\sum_{ k\in \mathbb Z_+}|\langle f, \psi^{(\nu)}_{jk}\rangle|^2=\|f\|_2^2 
	\quad \forall f\in L_2(V).
  \end{equation}
\end{theo}

Because of a well known property of the Parseval frames (see, e.g., \cite[Proposition~1.8.2]{NPS}), it follows from~(\ref{pars}) that
\begin{equation*}
  f=\sum_{\nu=1}^r\sum_{j\in \mathbb Z}\sum_{ k\in \mathbb Z_+}\langle f, \psi^{(\nu)}_{jk}\rangle \psi^{(\nu)}_{jk} 
  \quad \forall f\in L_2(V).
\end{equation*}

Obviously, if the system $\{\psi^{(\nu)}_{jk}\}_{\nu,j,k}$ is orthonormal and forms a Parseval frame, then it is an orthonormal basis.

A refinable function $\varphi$ is called {\it orthogonal} if the system $\{\varphi(\cdot\oplus h) \mid h\in H\}$ is 
orthonormal in $L_2(V)$.

\begin{theo}
  \label{11}
  Let $\varphi\in L_2(V)$. The system $\{\varphi(\cdot\oplus h)\mid h\in H\}$  is orthonormal in $L_2(V)$ if and only if
  $$
  \sum\limits_{h^*\in H^*}|\widehat\varphi(\cdot \oplus h^*)|^2=1\quad a.e.
  $$
\end{theo}
\begin{proof}
  The function $F=\sum\limits_{h^*\in H^*}|\widehat\varphi(\cdot \oplus h^*)|^2$ is $H^*$-periodic. Also, $F$ is summable on $U^*$ because
  $$
  \int\limits_{U^*}F\ d{\mu^*}=\sum\limits_{h^*\in H^*}\int\limits_{U^*\oplus h^*}|\widehat\varphi|^2\ d\mu^* = \int\limits_{V^{*}}|\widehat{\varphi}|^{2}\,d\mu^{*}<+\infty.
  $$
  If $ k\in\mathbb{Z}_+$, then
  \begin{align*}
    &\int\limits_{U^*}F\overline{W_k}\ d{\mu^*}=
      \int\limits_{U^*}\overline{W_{k}}\sum\limits_{h^*\in H^*}|\widehat\varphi(\cdot \oplus h^*)|^2\ d{\mu^*}
    \\
    &= \sum\limits_{h^*\in H^*}\int\limits_{U^*\oplus h^*}\overline{W_{k}}|\widehat\varphi|^2\ d{\mu^*}
      = \int\limits_{V^{*}}\overline{W_{k}}|\widehat\varphi|^2\ d{\mu^*}.
  \end{align*}
  Due to the Parseval equality and the formula~(\ref{Fourier_shift}) it follows that
  $$
  \int\limits_{U^*}F\overline{W_k}\ d{\mu^*}
  =\int\limits_{V^*}\widehat\varphi\overline{\chi(\gamma_{[k]},\cdot)\widehat\varphi}\ d{\mu^*}
  =\int\limits_V \varphi\overline{\varphi(\cdot\oplus\gamma_{[k]})}\ d\mu.
  $$
  Since the left hand side is nothing as the Fourier coefficient of $F$ with respect to 
  the system $\{W_k\}_{k\in \mathbb{Z}_+}$, which is  the orthonormal basis for $L_2(U^*)$, 
  this equality implies our statement.  
\end{proof}

\begin{theo}
  If a refinable function $\varphi$ with compact support is orthogonal, then its mask $m_0$ satisfies 
  \begin{equation}
    \sum\limits_{k=0}^{m-1}|m_0({\mathcal M^*}^{-1}(\omega\oplus s^*_k))|^2 =1 \quad \forall \omega\in V^*.
    \label{aaa1}
  \end{equation}
\end{theo}
\begin{proof}
  Using Theorem~\ref{11}, refinement equation~(\ref{refinement_eq_Fourier}) for $\varphi$ and  $H^*$-periodicity of the mask $m_0$, setting  $h^*=q \oplus{\mathcal M}^*\gamma$, where $\gamma\in H^*$, $q\in D^*$,  we have
  $$
  1=\sum\limits_{h^*\in H^*}|\widehat\varphi(\omega \oplus h^*)|^2=
  \sum\limits_{q\in D^*}\sum\limits_{\gamma\in H^*}|\widehat\varphi(\omega \oplus q\oplus \mathcal{M}^*\gamma)|^2
  $$
  $$
  =\sum\limits_{q\in D^*}|m_0({\mathcal{M}^*}^{-1}(\omega \oplus q))|^2\sum\limits_{\gamma\in H^*}|\widehat\varphi({\mathcal{M}^*}^{-1}(\omega \oplus q)\oplus \gamma)|^2=\sum\limits_{q\in D^*}|m_0({\mathcal{M}^*}^{-1}(\omega \oplus q))|^2,
  $$
  that is equivalent to~(\ref{aaa1}).
\end{proof}
\begin{theo}
  \label{300}
  Let $\varphi\in L_2(V)$ be an orthogonal compactly supported refinable function, $\widehat{\varphi}(\theta) = 1$. Then a wavelet system 
  $\{\psi^{(\nu)}_{jk} \mid j \in \mathbb{Z}, k \in \mathbb{Z}_{+}, \nu = 1, \ldots, m-1\}$ generated by $\varphi$ is an orthonormal basis in $L_{2}(V)$.
\end{theo}
\begin{proof}
  By theorem~\ref{thm:parseval-frame}, such a system $\left\{\psi_{jk}^{(\nu)}\right\}$ exists and forms a Parseval frame. So it suffices to show that the system $\left\{\psi_{jk}^{(\nu)}\right\}$ is orthonormal.
  
  First we will prove that for every $\nu, \nu'=1,\dots, m-1$ and for every $k, k'\in \mathbb Z$ 
  there holds
  $$
  \int\limits_{V} {\psi^{(\nu)}(\cdot\oplus \gamma_{[k]})} \overline{{ \psi^{(\nu')}(\cdot\oplus \gamma_{[k']})}}\ d\mu
  =\delta_{k,k'}\delta_{\nu,\nu'}.
  $$
  By Plancherel theorem, the left hand side equals
  $$
  \int\limits_{V^{*}}\widehat{\psi^{(\nu)}_{0k}}\overline{\widehat{\psi^{(\nu')}_{0k'}}}\,d\mu^{*}
  =
  \sum\limits_{\gamma\in H^*} \int\limits_{U^*\oplus \gamma}\widehat{\psi^{(\nu)}_{0k}}\overline{\widehat{\psi^{(\nu')}_{0k'}}}\,d\mu^{*}
  =
  \sum\limits_{\gamma\in H^*} \int\limits_{U^*} \widehat{\psi_{0k}^{(\nu)}}(\cdot\oplus \gamma)\overline{ \widehat{\psi_{0k'}^{(\nu')}}(\cdot\oplus \gamma)}\, d\mu^*.
  $$

  So it suffices to verify that
  \begin{equation}
    \label{333}
    \int\limits_{U^*}\sum\limits_{\gamma\in H^*} \widehat{\psi_{0k}^{(\nu)}}(\cdot\oplus \gamma)\overline{ \widehat{\psi_{0k'}^{(\nu')}}(\cdot\oplus \gamma)}\, d\mu^*=\delta_{k,k'}\delta_{\nu,\nu'} .
  \end{equation}
  Using formula~(\ref{Fourier_shift}) and $H^{*}$-periodicity of Walsh polynomials we have
  $$
  S:=\sum\limits_{\gamma\in H^*} \widehat{\psi_{0k}^{(\nu)}}(\cdot\oplus \gamma)\overline{ \widehat{\psi_{0k'}^{(\nu')}}(\cdot\oplus \gamma)}
  =
  \sum\limits_{\gamma\in H^*}\widehat{\psi^{(\nu)}}(\cdot \oplus \gamma)W_{k}\overline{\widehat{\psi^{(\nu')}}(\cdot \oplus \gamma)W_{k'}}.
  $$
  By the definition of wavelet functions, setting  $\gamma=q \oplus{\mathcal M}^*h$, where $h\in H^*$, $q\in D^*$, and using $H^*$-periodicity of $m_\nu$, we have
  \begin{align*}
    &S=\sum\limits_{\gamma\in H^*} m_\nu({ {\mathcal M}^*}^{-1}(\cdot\oplus\gamma))\widehat\varphi({ {\mathcal M}^*}^{-1}(\cdot\oplus \gamma))
      \overline{m_{\nu'}({ {\mathcal M}^*}^{-1}(\cdot\oplus\gamma))\widehat\varphi({ {\mathcal M}^*}^{-1}(\cdot\oplus\gamma)} W_k\overline{ W_{k'}}
    \\
    &= \sum\limits_{q\in D^*} m_\nu({ {\mathcal M}^*}^{-1}(\cdot\oplus q))\overline{m_{\nu'}({ {\mathcal M}^*}^{-1}(\cdot \oplus q))} )W_k\overline{ W_{k'}} \sum\limits_{h\in H^*}|\widehat\varphi({ {\mathcal M}^*}^{-1 }(\cdot\oplus q)\oplus h)|^2.
  \end{align*}
  Due to orthogonality of $\varphi$,  Theorem~\ref{11}  and (\ref{wav}), this yields
  $$
  S=\sum\limits_{q\in D^*} m_\nu({{\mathcal M}^*}^{-1}(\cdot\oplus q))\overline{m_{\nu'}
	({{\mathcal M}^*}^{-1}(\cdot \oplus q))}W_k\overline{ W_{k'}}
  =W_k\overline{ W_{k'}}\delta_{\nu,\nu'}
  $$
  which together with orthonormality of Walsh functions on $U^*$ implies~(\ref{333}). 

  It remains to check that the functions $\psi^{(\nu)}_{jk}$, $\psi^{(\nu')}_{j'k'}$ are orthogonal for all 
  $\nu, \nu', k, k'$ whenever $j\ne j'$. First, let us  prove that any function $\psi^{(\nu)}_{0k}$ is orthogonal 
  to the space $V_0:=\overline{\mathrm{span}\{\varphi(\cdot\oplus h)\mid  h\in H\}}$.  Again setting $\gamma=q \oplus{\mathcal M}^*h$, where $h\in H^*$, $q\in D^*$, and using $H^*$-periodicity of $m_\nu$, we have
  \begin{align*}
    &\int\limits_V \psi^{(\nu)}(\cdot\oplus \gamma_{[k]})\overline{\varphi(\cdot\oplus \gamma_{[k']})} d\mu=
      \int\limits_{V^*} \widehat{\psi^{(\nu)}}W_{k}\overline{\widehat\varphi W_{k'}}d\mu^*
    \\
    &=\sum\limits_{\gamma\in H^*} \int\limits_{U^*\oplus \gamma}m_\nu({{\mathcal M}^*}^{-1}\cdot) W_{k}\overline{
      m_0({{\mathcal M}^*}^{-1}\cdot)W_{k'}}|\widehat\varphi({{\mathcal M}^*}^{-1}\cdot)|^2 d\mu^*
    \\
    &=\int\limits_{U^*}W_{k}\overline{W_{k'}}\sum_{q\in D^{*}}m_\nu({{\mathcal M}^*}^{-1}(\cdot\oplus q))
      \overline{	m_0({{\mathcal M}^*}^{-1}(\cdot\oplus q))} K_{q} \ d\mu^*,
  \end{align*}
  where
  $$
  K_{q}(\omega)=\sum_{h\in H^*}\big|\widehat\varphi({{\mathcal M}^*}^{-1}(\omega\oplus q)\oplus h)\big|^2.
  $$
  But $|K_{q}(\omega)| = 1$ for almost all $\omega$ because of orthogonality of $\varphi$ and Theorem~\ref{11}, 
  which yields
  \begin{align*}
    &\left|\int\limits_V \psi^{(\nu)}(\cdot\oplus \gamma_{[k]})\overline{\varphi(\cdot\oplus \gamma_{[k']})} d\mu\right|
      \leq\int\limits_{U^*}\Bigg|\sum_{q \in D^{*}}m_\nu({{\mathcal M}^*}^{-1}(\cdot\oplus q))
      \overline{	m_0({{\mathcal M}^*}^{-1}(\cdot\oplus q))} \Bigg| \ d\mu^*,
  \end{align*}
  and the latter integral is equal to zero due to (\ref{wav}). 

  So, we proved that $\psi^{(\nu)}_{0k}$ is orthogonal to  $V_0$. Let us show that any function 
  $\psi^{(\nu')}_{{-j}k'}$, $j>0$, belongs to $V_0$. Since, obviously,
  $$
  \psi^{(\nu')}(x)=\sum_{k=0}^\infty m\beta_k \varphi({\mathcal M}x \oplus \gamma_{[k]})
  $$
  where $\beta_k$ are the coefficients of $m_{\nu'}$ with respect to system $\left\{W_{k}\right\}_{k\in \mathbb{Z}_+}$, we have 
  $$
  \psi^{(\nu')}({\mathcal M}^{-j}x\oplus \gamma_{[k']})=\sum_{k=0}^\infty m\beta_k \varphi({\mathcal M}^{1-j}x\oplus{\mathcal M}\gamma_{[k']} \oplus \gamma_{[k]}),
  $$
  which yields $\psi^{(\nu')}_{(-j)k'}\in V_0$ by~(\ref{refinable}) and definition of $V_{0}$. Thus, any two functions $\psi^{(\nu)}_{0k}$ and $\psi^{(\nu')}_{(-j)k'}$ 
  are orthogonal. In a similar way, any two functions $\psi^{(\nu)}_{jk}$ and $\psi^{(\nu')}_{j'k'}$ with $j\ne j'$ are orthogonal, which was to be proved.
\end{proof}

Now we are interested in a method for construction of compactly supported orthogonal refinable functions.
It follows from the proof of Theorem~4 in~\cite{BS} that every Walsh polynomial $m_0$, satisfying~(\ref{aaa}) and such that $m_0(\theta)=1$, is the mask of a compactly supported refinable function $\varphi$ whose Fourier transform is defined by
\begin{equation}
  \widehat\varphi=\prod_{j=1}^\infty m_0({{\mathcal M}^*}^{-j}\cdot).
  \label{prod}
\end{equation}

So, we know how to construct a compactly supported refinable functions $\varphi$  on the basis of any appropriate Walsh polynomial $m_0$. Let us discuss how to provide its orthogonality. 

It will be convenient for us to write an element $\omega=\{\omega_j\}_{j=-N}^\infty$ from $ V^*$ in the following form
$$
\omega=\omega_{-N}\omega_{-N+1}\dots \omega_0\mbox{\boldmath $,$\,}  \omega_1 \omega_2 \omega_2 \dots,  \quad \omega_j\in D^*.
$$
In the case, where  $\omega_j=\mathbf{0}$ for all $j>n$,  we will write
$$
\omega=\omega_{-N}\omega_{-N+1}\dots \omega_0\mbox{\boldmath $,$\,} \omega_1\omega_2\dots \omega_n,
$$
and in the case, where $\omega_j=\mathbf{0}$ for all $j>0$ (i.e., $\omega\in H^*$),  we will write
$$
\omega=\omega_{-N}\omega_{-N+1}\dots \omega_0.
$$

Denote by $\mathcal{M}_0^{(n)}$ the set of Walsh polynomials of order $n$ such that $m_0(\theta)=1$.
Since the 
Walsh functions $W_k$ are $H^*$-periodic,  constant on each set $U^*_{n,s}$,  $s \in\mathbb{Z}_+$, $0 \leq k \leq m^{n}-1$, 
for every $m_0\in \mathcal{M}_0^{(n)}$ and for every $\omega\in V^*$, we have 
$$
m_0(\omega)= m_0(\mathbf{0}\mbox{\boldmath $,$\,}  \omega_1\ldots \omega_n)=:b(\omega_1\ldots \omega_n).
$$
In particular, it follows that  only a finite number of the factors in  product~(\ref{prod}) is not  identical~1.  
Also, it is clear that to define $m_0$ on $V^*$, one needs  to determine all values $b(\omega_1\ldots \omega_n)$.

To construct $m_0\in \mathcal{M}_0^{(n)}$, $n\geq 2$, that is the mask of a compactly supported orthogonal refinable function $\varphi$, we  choose a positive integer $r\ge n$ and a set
\begin{equation}
  \label{301}
  \{\xi_0,\dots, \xi_{r}\}\subset D^*\setminus\{s^*_0\},
\end{equation}
such that
\begin{equation}
  \label{302}
  (\xi_k,\dots, \xi_{k+n-2})\ne (\xi_{k'},\dots, \xi_{k'+n-2}) \quad \forall k\ne k'.
\end{equation} 
It is clear that the maximum possible $r$ depends on $n$ and $m$. For example, if $n=m=3$, then the  maximal $r$   equals 4, and if $m=3$, $n=4$, then one can choose $r=9$.

Set 
$$
b(\xi_k\dots \xi_{k+n-1})=\pm 1
$$
for every $k=0,\dots, r-n + 1$; if $(\omega_2,\dots ,\omega_{n})$ does not coincide with
$(\xi_k,\dots,  \xi_{k+n-2})$ for some $k=1,\dots, r-n+2$,  
we set
$$
b(\mathbf{0} \omega_2\dots \omega_{n})=\pm 1, \quad b(\mathbf{0}\ldots \mathbf{0}) = 1,
$$
and for all other sets $(\omega_1, \dots, \omega_{n})$ we vanish the corresponding $b(\omega_1 \dots\omega_{n})$. 
So, we defined a Walsh polynomial $m_0\in \mathcal{M}_0^{(n)}$. Obviously, we have 
\begin{equation}
  \sum\limits_{k=0}^{m-1}\Big|m_0({\bf 0}\mbox{\boldmath $,$\,} s_k^*\omega_2\dots\omega_n)\Big|^2 =1 \quad\forall (\omega_2,\dots, \omega_n),\ \omega_j\in D^*,
  \label{33.20}
\end{equation}
that is equivalent to equality~(\ref{aaa1}). Hence, $m_0$ is the mask of a  
compactly supported refinable function~$\varphi$, defined by~(\ref{prod}). 
We will say that such a function $\varphi$ and its mask $m_0$ are {\it associated with the set} $\{\xi_0,\dots, \xi_{r}\}$. 

\begin{theo}
  Let $n, r\in \mathbb N$, $r\geq n\geq 2$. If $\varphi$ is the refinable function associated with 
  a set of digits $\{\xi_0,\dots, \xi_{r}\}$, 
  satisfying~(\ref{301}), (\ref{302}), then $\varphi$ is orthogonal, $\varphi \in \mathcal{S}^{(n-1)}_{r-n+2}$, and
  \begin{equation}
    \begin{split}
      \widehat{\varphi}(\omega)
      &=
        \sum\limits_{k=1}^{r-n+2}\mathbf{1}_{\xi_{0}\ldots\xi_{k-1}\mbox{\boldmath $,$\,}\xi_{k}\ldots\xi_{k+n-2}\oplus U^{*}_{n-1}}(\omega)\prod\limits_{j=1}^{k+n-1}b(\xi_{k-j} \ldots\xi_{k+n-1-j})
      \\
      +
      &\mathbf{1}_{Q\oplus U^{*}_{n-1}}(\omega)\prod\limits_{j=1}^{n-1}b(\mathbf{0} \ldots \mathbf{0} \omega_{1} \ldots \omega_{n-j}),
    \end{split}
    \label{eq:phihat_formula}
  \end{equation}
  where $\xi_{l} = \mathbf{0}$ for $l \leq -1$,
  $$
  Q = \left\{\mathbf{0}\mbox{\boldmath $,$\,} \alpha_{1}\ldots \alpha_{n-1} \mid \alpha_{1} \ldots \alpha_{n-1} \neq \xi_{k} \ldots \xi_{k+n-2} \quad \forall k = 1, 2, \ldots, r-n+2\right\}.
  $$
\end{theo}
\begin{proof}
  Due to Theorem~\ref{11}, to prove the orthogonality of $\varphi$, it suffices to verify that for every $\omega\in V^*$  there exists a unique $h\in H^*$  such that 
  $\widehat\varphi(\omega \oplus h)\ne 0$, and moreover $|\widehat\varphi(\omega \oplus h)|=1$. Obviously, this is equivalent to the following statement: for every $\omega\in U^*$ there exists a unique $h\in H^*$ such that 
  $\widehat\varphi(\omega \oplus h)\ne 0$, and moreover $|\widehat\varphi(\omega \oplus h)|=1$. 

  Let $h= h_{-N}\dots h_0$. 
  Suppose that for some $k=1,\dots, r-n+2$ we have
  $$
  \omega= \mathbf{0}\mbox{\boldmath $,$\,} \xi_k\dots \xi_{k+n-2}\omega_{n}\omega_{n+1}\dots,
  $$
  and $\widehat{\varphi}(\omega \oplus h) \neq 0$. By formula~(\ref{prod})
  \begin{align*}
    &\widehat{\varphi}(\omega \oplus h)
      =
      \prod\limits_{j = 1}^{\infty}m_{0}\left(\mathcal{M}^{-j}(\omega \oplus h)\right)
    \\
    &=
      b(h_{0}\xi_{k} \ldots \xi_{k + n-2})b(h_{-1}h_{0} \xi_{k} \ldots \xi_{k + n - 3})\cdot \ldots \cdot b(h_{-j} \ldots h_{0}\xi_{k} \ldots \xi_{k+n-j-2}) \cdot \ldots .
  \end{align*}
  The factor $b(h_0 \xi_k\dots \xi_{k+n-2})$ is nonzero only if 
  $h_0=\xi_{k-1}$. Similarly, if $1\leq j\le k-1$ and $h_0=\xi_{k-1}, h_{-1}=\xi_{k-2}, \dots, h_{-j+1}=\xi_{k-j}$, then 
  $$
  b(h_{-j}h_{-j+1}\dots h_0 \xi_k\dots \xi_{k+n-j-2})=
  b(h_{-j}\xi_{k-j}\dots \xi_{k+n-j-2})
  $$
  is nonzero only if $h_{-j}=\xi_{k-j-1}$. Hence, $h_{-j+1}=\xi_{k-j}$ for $j=1,\dots, k$.

  Since $(\xi_0,\dots, \xi_{n-2}) \neq (\xi_k,\dots, \xi_{k+n-2})$  for any $k=1,\dots, r-n+2$, we have
  $$
  b(h_{-k}h_{-k+1} \ldots)=b(h_{-k} \xi_0\dots \xi_{n-2}),
  $$ 
  that is nonzero only if $h_{-k}=\mathbf{0}$. Then
  $$
  b(h_{-k-1}h_{-k} \ldots)=b(h_{-k-1}\mathbf{0} \xi_0\dots \xi_{n-3})
  $$ 
  is nonzero only if $h_{-k-1}=\mathbf{0}$.  Similarly, we conclude that only if $h_j=\mathbf{0}$ for all $j=-N, \ldots, -k-2$, 
  we can provide nonzero $\widehat\varphi(\omega\oplus h)$.  Thus, there exists a unique vector $h=\xi_0\dots \xi_{k-1}$ such that
  $\widehat\varphi(\omega\oplus h)\ne 0$, and moreover $|\widehat\varphi(\omega\oplus h)|=1$.

  We can write for $k = 1, \ldots, r - n + 2$ that
  \begin{equation}
    \label{eq:restriction_not_Q}
    \begin{split}
      &\widehat{\varphi}|_{H^{*} \oplus \mathbf{0}\mbox{\boldmath $,$\,}\xi_{k} \ldots \xi_{k+n-2} \oplus U^{*}_{n-1}}
      \\
      &=
        \mathbf{1}_{\xi_{0} \ldots \xi_{k-1}\mbox{\boldmath $,$\,} \xi_{k} \ldots \xi_{k+n-2} \oplus U^{*}_{n-1}}
        \prod\limits_{j=1}^{k+n-1}b(\xi_{k-j} \ldots \xi_{k+n-1-j}),
    \end{split}
  \end{equation}
  where $\xi_{l} = \mathbf{0}$ for $l \leq -1$.

  Now we assume that $\omega \in Q \oplus U^{*}_{n-1}$. Then the factor
  $b(h_{0} \omega_1\dots \omega_{n-1})$ 
  is nonzero only if $h_0=\mathbf{0}$. By the same arguments, for any $j\in\{1,\dots, n-2\}$, the factor
  $$
  b(h_{-j}h_{-j+1} \ldots) = b(h_{-j}\mathbf{0}\dots \mathbf{0}  \omega_1\dots \omega_{n-j-1})
  $$ 
  is nonzero only if $h_{-j}=\mathbf{0}$. Similarly, all other digits of $h$ must be zero to provide
  nonzero $\widehat\varphi(\omega\oplus h)$.  Hence, there exists a unique vector $h = \theta$ such that 
  $\widehat\varphi(\omega\oplus h)\ne 0$, and moreover $|\widehat\varphi(\omega\oplus h)|=1$.

  We can write
  \begin{equation}
    \label{eq:restriction_Q}
    \widehat{\varphi}|_{H^{*}\oplus Q \oplus U^{*}_{n-1}}
    = \mathbf{1}_{Q\oplus U^{*}_{n-1}}\prod\limits_{j=1}^{n-1}b(\mathbf{0} \ldots \mathbf{0} \omega_{1} \ldots \omega_{n-j}).
  \end{equation}

  Relations~(\ref{eq:restriction_not_Q}) and~(\ref{eq:restriction_Q}) lead to formula~(\ref{eq:phihat_formula}), which implies that $\widehat{\varphi} \in \mathcal{S}^{(r-n+2)}_{n-1}$. By theorem~\ref{p8.2} the belonging $\varphi \in \mathcal{S}^{(n-1)}_{r-n+2}$ holds true.  
\end{proof}

Now we know how to construct  a compactly supported orthogonal refinable function $\varphi$ with a mask $m_0\in \mathcal{M}_0^{(n)}$. Let us discuss the construction of wavelet systems generated by $\varphi$. 

For every fixed $(\omega_2,\dots,\omega_n)$, we denote $m_0({\bf 0}\mbox{\boldmath $,$\,} s_k^*\omega_2\dots\omega_n)$ by $\alpha_{0 k}$  
and find  numbers $\alpha_{\nu k}$, $\nu=1, \dots, m-1$, $ k = 0,\dots, m-1$, such that the matrix  $(\alpha_{\nu k})_{\nu,k=0}^{m-1}$ is unitary. Note that this is possible due to~(\ref{33.20})  and the Householder transform (see, e.g., \cite{NPS}, p. 327). 
Setting 
$$
m_\nu ({\bf 0}\mbox{\boldmath $,$\,} s_k^*\omega_2\dots\omega_n):= \alpha_{\nu k},
$$
we get wavelet masks $m_\nu$ which are Walsh polynomials of order $n$, and wavelet functions defined by
$$
\widehat {\psi^{(\nu)}}= m_\nu({\mathcal M^*}^{-1}\cdot)\widehat \varphi({\mathcal M^*}^{-1}\cdot),\quad \nu=1,\dots, m-1. 
$$
The corresponding wavelet system $\{\psi^{(\nu)}_{jk}\}_{j,k,\nu}$ is generated by $\varphi$  (see~\cite[Theorem~5]{BS} for details). Due to Theorem~\ref{300}, such a system $\{\psi^{(\nu)}_{jk}\}_{j,k,\nu}$ forms an orthonormal wavelet basis.

\bigskip
\textbf{Example 1.} Let $ M=   \left( {\begin{array}{cc}
  2 & 0 \\
  1 & 2\\
\end{array} } \right)$, $n=3$. Then we have $m=4$,  $ M^*=   \left( {\begin{array}{cc}
  2 & 1 \\
  0 & 2\\
\end{array} } \right)$, and the set of vectors 
$$
s_0^*={\bf 0}=(0,0), s_1^*=(0,1), s_2^*=(1,0), s_3^*=(1,1)
$$
can be taken as  $D^*$. To construct $m_0\in \mathcal{M}_0^{(3)}$ we choose
$$
\xi_0=s_1^*,\ \xi_1=s_1^*,\  \xi_2=s_2^*,\  \xi_3=s_1^*
$$
and, according to our method of construction, $m_0({\bf 0}\mbox{\boldmath $,$\,}s_1^* s_1^* s_2^*)=1, m_0({\bf 0}\mbox{\boldmath $,$\,} s_1^* s_2^* s_1^*)=~1$,  and 
$m_0({\bf 0}\mbox{\boldmath $,$\,}s_l^* s_1^* s_2^*)=m_0({\bf 0}\mbox{\boldmath $,$\,} s_l^* s_2^* s_1^*)=0$ whenever $l=0,2,3$. Similarly, if 
$(\omega_1,  \omega_2)$ is different from $(s_1^*, s_2^*)$ and from $(s_2^*, s_1^*)$, then 
$m_0({\bf 0}\mbox{\boldmath $,$\,} s_l^* \omega_1 \omega_2)=0$ for $l=1,2,3$ and $m_0({\bf 0}\mbox{\boldmath $,$\,} \mathbf{0} \omega_1 \omega_2)=1$.

By formula~(\ref{eq:phihat_formula}) the refinable function $\varphi$ is given by its Fourier transform
$$
\widehat\varphi
=\mathbf{1}_{s_{1}^{*}\mbox{\boldmath $,$\,}s_{1}^{*}s_{2}^{*} \oplus U^{*}_{2}}
+\mathbf{1}_{s_{1}^{*}s_{1}^{*}\mbox{\boldmath $,$\,}s_{2}^{*}s_{1}^{*} \oplus U^{*}_{2}}
+\mathbf{1}_{Q \oplus U^{*}_{2}},
$$
where 
$Q=\left\{\mathbf{0}\mbox{\boldmath $,$\,}\alpha_{1}\alpha_{2} \mid \alpha_{1}\alpha_{2} \not\in \left\{s_{1}^{*}s_{2}^{*}, s_{2}^{*}s_{1}^{*}\right\}\right\}$. Notice also, that $\widehat\varphi, \varphi \in \mathcal{S}_{2}^{(2)}$.

To get wavelet functions, it remains to determine wavelet masks $m_\nu$. Obviously, to provide~(\ref{wav}),
we can set $m_\nu(\mathbf{0}\mbox{\boldmath $,$\,} s^*_l\omega_1\omega_2)=\delta_{l\nu}$, $\nu=1,2,3$, whenever
$(\omega_1,  \omega_2)$ is different from $(s_1^*, s_2^*)$ and from $(s_2^*, s_1^*)$;
\begin{align*}
  &m_1({\bf 0}\mbox{\boldmath $,$\,} s_0^* s_2^* s_1^*)=m_1({\bf 0}\mbox{\boldmath $,$\,} s_0^* s_1^* s_2^*)=1;
    \  m_1({\bf 0}\mbox{\boldmath $,$\,} s_l^* s_2^* s_1^*)=m_1({\bf 0}\mbox{\boldmath $,$\,} s_l^* s_1^* s_2^*)=0,\  l=1,2,3;
  \\
  &m_\nu({\bf 0}\mbox{\boldmath $,$\,} s_l^* s_2^* s_1^*)=m_\nu({\bf 0}\mbox{\boldmath $,$\,} s_l^* s_1^* s_2^*)=\delta_{\nu l}, \nu=2,3, \ l=0,1,2,3.
\end{align*}

\textbf{Example 2.} Let $M = \begin{pmatrix}
  2 & 1\\
  -1 & 1
\end{pmatrix}$, $n = 3$. Then $m = 3$, $M^{*} = \begin{pmatrix}
  2 & -1\\
  1 & 1
\end{pmatrix}$, and, for instance, $s_{0}^{*} = (0, 0), s_{1}^{*}= (0, 1), s_{2}^{*} = (1, 1)$. Choose $r = 4$,
$$
(\xi_{0}, \xi_{1}, \xi_{2}, \xi_{3}, \xi_{4}) = (s_{1}^{*}, s_{2}^{*}, s_{2}^{*}, s_{1}^{*}, s_{1}^{*}).
$$
We define the mask $m_{0}$ on $U^{*}$ by its values on a set $\left\{\mathbf{0}\mbox{\boldmath $,$\,}\omega_{1}\omega_{2}\omega_{3} \mid \omega_{j} \in D^{*}\right\}$ as
$$
m_{0}(\mathbf{0}\mbox{\boldmath $,$\,}\omega_{1}\omega_{2}\omega_{3})
= \begin{cases}
  \left.
  \begin{cases}
    1,	&\text{if } \omega_{2}\omega_{3} \not\in \left\{s_{2}^{*}s_{2}^{*}, s_{2}^{*}s_{1}^{*}, s_{1}^{*}s_{1}^{*}\right\}\\
    0,	&\text{otherwise} \\
  \end{cases}\right\},	&\text{if } \omega_{1} = \mathbf{0},\\
  \left.
  \begin{cases}
    1,	&\text{if } \omega_{1}\omega_{2}\omega_{3} \in \left\{s_{1}^{*}s_{2}^{*}s_{2}^{*}, s_{2}^{*}s_{2}^{*}s_{1}^{*}, s_{2}^{*}s_{1}^{*}s_{1}^{*}\right\}\\
    0,	&\text{otherwise} \\
  \end{cases}\right\},	&\text{if } \omega_{1} \neq \mathbf{0}.\\
\end{cases}
$$

By formula~(\ref{eq:phihat_formula}) the refinable function $\varphi$ is given by its Fourier transform
$$
\widehat{\varphi}
= \mathbf{1}_{s_{1}^{*}\mbox{\boldmath $,$\,}s_{2}^{*}s_{2}^{*}\oplus U^{*}_{2}}
+ \mathbf{1}_{s_{1}^{*}s_{2}^{*}\mbox{\boldmath $,$\,}s_{2}^{*}s_{1}^{*}\oplus U^{*}_{2}}
+ \mathbf{1}_{s_{1}^{*}s_{2}^{*}s_{2}^{*}\mbox{\boldmath $,$\,}s_{1}^{*}s_{1}^{*}\oplus U^{*}_{2}}
+ \mathbf{1}_{Q \oplus U^{*}_{2}},
$$
where $Q = \left\{\mathbf{0}\mbox{\boldmath $,$\,}\alpha_{1}\alpha_{2} \mid \alpha_{1}\alpha_{2} \not\in \left\{s_{2}^{*}s_{2}^{*}, s_{2}^{*}s_{1}^{*}, s_{1}s_{1}^{*}\right\}\right\}$.

If $\omega_{2}\omega_{3} \not\in \left\{s_{2}^{*}s_{2}^{*}, s_{2}^{*}s_{1}^{*}, s_{1}^{*}s_{1}^{*}\right\}$, define
$$
m_{\nu}(\mathbf{0}\mbox{\boldmath $,$\,}s_{l}^{*}\omega_{2}\omega_{3}) = \delta_{\nu l}, \quad \nu = 1, 2, \quad l = 0, 1, 2.
$$
When $\omega_{2}\omega_{3} \in \left\{s_{2}^{*}s_{2}^{*}, s_{2}^{*}s_{1}^{*}, s_{1}^{*}s_{1}^{*}\right\}$, we take the values of $m_{\nu}$ from the following tables
$$
\begin{array}{r|ccc}
  \hline
  \omega_{1} & s_{0}^{*} & s_{1}^{*} & s_{2}^{*}\\
  \hline
  m_{0}(\mathbf{0}\mbox{\boldmath $,$\,}\omega_{1}s_{2}^{*}s_{2}^{*}) & 0 & 1 & 0\\
  \hline
  m_{1}(\mathbf{0}\mbox{\boldmath $,$\,}\omega_{1}s_{2}^{*}s_{2}^{*}) & 1 & 0 & 0\\
  \hline
  m_{2}(\mathbf{0}\mbox{\boldmath $,$\,}\omega_{1}s_{2}^{*}s_{2}^{*}) & 0 & 0 & 1\\
  \hline
\end{array}
\quad
\begin{array}{r|ccc}
  \hline
  \omega_{1} & s_{0}^{*} & s_{1}^{*} & s_{2}^{*}\\
  \hline
  m_{0}(\mathbf{0}\mbox{\boldmath $,$\,}\omega_{1}s_{2}^{*}s_{1}^{*}) & 0 & 0 & 1\\
  \hline
  m_{1}(\mathbf{0}\mbox{\boldmath $,$\,}\omega_{1}s_{2}^{*}s_{1}^{*}) & 0 & 1 & 0\\
  \hline
  m_{2}(\mathbf{0}\mbox{\boldmath $,$\,}\omega_{1}s_{2}^{*}s_{1}^{*}) & 1 & 0 & 0\\
  \hline
\end{array}
$$
$$
\begin{array}{r|ccc}
  \hline
  \omega_{1} & s_{0}^{*} & s_{1}^{*} & s_{2}^{*}\\
  \hline
  m_{0}(\mathbf{0}\mbox{\boldmath $,$\,}\omega_{1}s_{1}^{*}s_{1}^{*}) & 0 & 0 & 1\\
  \hline
  m_{1}(\mathbf{0}\mbox{\boldmath $,$\,}\omega_{1}s_{1}^{*}s_{1}^{*}) & 0 & 1 & 0\\
  \hline
  m_{2}(\mathbf{0}\mbox{\boldmath $,$\,}\omega_{1}s_{1}^{*}s_{1}^{*}) & 1 & 0 & 0\\
  \hline
\end{array}
$$
The condition~(\ref{wav}) is satisfied, thus the relations $\widehat{\psi^{(\nu)}}(\omega) = m_{\nu}(\mathcal{M}^{*-1}\omega)\widehat{\varphi}(\mathcal{M}^{*-1}\omega)$, $\nu = 1, 2$, define wavelet functions.

\bigskip

{\bf M. Babushkin}

St. Petersburg State University;

ITMO University

E-mail: m.v.babushkin@yandex.ru

\smallskip
{\bf M. Skopina}

St. Petersburg State University;

Higher School of Economics, St. Petersburg 

Regional Mathematical Center of Southern Federal University

E-mail: skopinama@gmail.com

\end{document}